\documentclass[graybox]{svmult}

\usepackage{amssymb}
\usepackage{amsmath}
\usepackage{calrsfs}

\usepackage{enumerate}
\usepackage{enumitem}

\def\bfit{\bfseries\itshape}

\def\ZM{{\mathbb{Z}}}
\def\AC{{\mathcal{A}}}
\def\DC{{\mathcal{D}}}
\def\PC{{\mathcal{P}}}
\def\RC{{\mathcal{R}}}
\def\VC{{\mathcal{V}}}
\def\Arm{{\mathrm{A}}}  
\def\D{\Delta}
\def\G{\Gamma}
\def\d{\delta}
\def\r{\rho}
\def\s{\sigma}
\def\t{\tau}
\newcommand{\fS}{{\mathfrak{S}}}
\newcommand{\bH}{{\mathcal{H}}}
\newcommand{\cI}{{\mathcal{I}}}
\newcommand{\cT}{{\mathfrak{T}}}
\newcommand{\Ind}{{\mathrm{Ind}}}
\newcommand{\prI}{{\mathrm{pr}}}
\newcommand{\longiso}{\stackrel{\sim}{\longrightarrow}}
\def\to{\rightarrow}
\def\longto{\longrightarrow}
\def\eqna{\begin{eqnarray*}}
\def\endeqna{\end{eqnarray*}}
\def\itemth#1{\item[${\mathrm{(#1)}}$]}
\renewcommand{\le}{\leqslant}
\renewcommand{\ge}{\geqslant}
\renewcommand{\leq}{\leqslant}
\renewcommand{\geq}{\geqslant}
\def\equat{\refstepcounter{thm}$$~}
\def\endequat{\leqno{{(\thesection.\arabic{thm})}}~$$}
\renewcommand{\seccountergap}{.\hskip\betweenumberspace}
\def\abstractname{Abstract.}
\def\ackname{Acknowledgements.}%

\spnewtheorem{thm}{Theorem}[section]{\bfseries}{\itshape}
\spnewtheorem{lem}[thm]{Lemma}{\bfseries}{\itshape}
\spnewtheorem{conj}[thm]{Conjecture}{\bfseries}{\itshape}
\spnewtheorem{cor}[thm]{Corollary}{\bfseries}{\itshape}
\spnewtheorem{prop}[thm]{Proposition}{\bfseries}{\itshape}

\spnewtheorem{exmp}[thm]{Example}{\bfseries}{}
\spnewtheorem{defn}[thm]{Definition}{\bfseries}{}
\spnewtheorem{abschnitt}[thm]{}{\bfseries}{}
\spnewtheorem{rem}[thm]{Remark}{\bfseries}{}

\begin{document}
\motto{To David Vogan on the occasion of his $60$th birthday}
\title*{Hecke algebras with unequal parameters and Vogan's left cell
invariants}

\author{C\'edric Bonnaf\'e and Meinolf Geck}
\institute{C\'edric Bonnaf\'e \at 
Institut de Math\'ematiques et de Mod\'elisation de Montpellier 
(CNRS: UMR 5149), Universit\'e Montpellier 2,
Case Courrier 051,
Place Eug\`ene Bataillon,
34095 Montpellier Cedex,
France, \\
\email{cedric.bonnafe@univ-montp2.fr} \and
Meinolf Geck \at 
Fachbereich Mathematik, IAZ - Lehrstuhl f\"ur Algebra, 
Universit\"at Stuttgart, Pfaf\-fen\-wald\-ring 57, 70569 Stuttgart, 
Germany,\\
\email{meinolf.geck@mathematik.uni-stuttgart.de}}

\maketitle

\abstract{In 1979, Vogan introduced a generalised $\tau$-invariant for 
characterising primitive ideals in enveloping algebras. Via a known
dictionary this translates to an invariant of left cells in the sense
of Kazhdan and Lusztig. Although it is not a complete invariant, it is 
extremely useful in describing left cells. Here, we propose a general 
framework for defining such invariants which also applies to Hecke
algebras with unequal parameters.}

\keywords{Coxeter groups, Hecke algebras, Kazhdan--Lusztig cells, \\ 
MSC 2010: Primary 20C08, Secondary 20F55}

\section{Introduction}
Let $W$ be a finite Weyl group. Using the corresponding generic
Iwahori--Hecke algebra and the ``new'' basis of this algebra introduced
by Kazhdan and Lusztig \cite{KL}, we obtain partitions of $W$ into left,
right and two-sided cells. Analogous notions originally arose in the
theory of primitive ideals in enveloping algebras; see Joseph \cite{jos}.
This is one of the sources for the interest in knowing the cell partitions
of $W$; there are also deep connections \cite{LuICM} with representations of 
reductive groups, singularities of Schubert cells and the geometry of
unipotent classes. Vogan \cite{voga}, \cite{voga1} introduced invariants 
of left cells which are computable in terms of certain combinatorially
defined operators $T_{\alpha\beta}$, $S_{\alpha\beta}$ where $\alpha,\beta$ 
are adjacent simple roots of $W$. In the case where $W$ is the symmetric 
group $\fS_n$, these invariants completely characterise the left cells; 
see \cite[\S 5]{KL}, \cite[\S 6]{voga}. Although Vogan's invariants are not 
complete invariants in general, they have turned out to be extremely
useful in describing left cells.

Now, the Kazhdan--Lusztig cell partitions are not only defined and
interesting for finite Weyl groups, but also for affine Weyl groups
and Coxeter groups in general; see, e.g., Lusztig \cite{Lu1},
\cite{Lusztig03}. Furthermore, the original theory was extended by
Lusztig \cite{Lusztig83} to allow the possibility of attaching weights to
the simple reflections. The original setting then corresponds to the case
where all weights are equal to $1$; we will refer to this case as the
``equal parameter case''. Our aim here is to propose analogues of Vogan's 
invariants which work in general, i.e., for arbitrary Coxeter groups and 
arbitrary (positive) weights.

In Sections~\ref{sec1} and \ref{sec:parabolic} we briefly recall the 
basic set-up concerning Iwahori--Hecke algebras, cells in the sense 
of Kazhdan and Lusztig, and the concept of induction of cells. In
Section~\ref{sec3} we introduce the notion of {\it left cellular maps};
a fundamental example is given by the Kazhdan--Lusztig $*$-operations.
In Section~\ref{seceq}, we discuss the equal parameter case and Vogan's 
original definition of a generalised $\tau$-invariant. As this definition
relied on the theory of primitive ideals, it only applies to finite Weyl
groups. In Theorem~\ref{vogan1}, we show that this works for arbitrary
Coxeter groups satisfying a certain boundedness condition. (A similar
result has also been proved by Shi \cite[4.2]{shi}, but he uses a 
definition slightly different from Vogan's; our argument seems to be more 
direct.) In Sections~\ref{sec2} and~\ref{sec6}, we develop an abstract 
setting for defining such invariants; this essentially relies on the 
concept of induction of cells and is inspired by Lusztig's method of 
{\it strings} \cite[\S 10]{Lu1}. As a bi-product of our approach, we 
obtain that the $*$-operations also work in the unequal parameter case. 
We conclude by discussing examples and stating open problems.

\medskip
\noindent{\bf Remark.}  In~\cite[Cor.~6.2]{bonnafe b}, the first author 
implicitly assumed that the results on the Kazhdan--Lusztig $*$-operations 
\cite[\S 4]{KL} also hold in the unequal parameter context~---~which was
a serious mistake at the time. Corollary~\ref{star1} below justifies 
{\it a posteriori} this assumption.

\medskip
\noindent{\bf Notation.}  
We fix a Coxeter system $(W,S)$ and we denote by $\ell : 
W \to \ZM_{\ge 0}$ the associated length function. We also fix a totally 
ordered abelian group $\AC$. We use an exponential notation for the group 
algebra $A=\ZM[\AC]$: 
\[A=\mathop{\oplus}_{a \in \AC} \ZM v^a \qquad \mbox{where}\qquad  
v^a v^{a'}= v^{a+a'} \mbox{ for all $a$, $a' \in \AC$}.\] 
We write $\AC_{\leq 0}:=\{\alpha\in \AC\mid \alpha \leq 0\}$ and
$A_{\leq 0}:=\mathop{\oplus}_{a \in \AC_{\leq 0}} \ZM v^a$; the 
symbols $\AC_{\geq 0}$, $A_{\geq 0}$ etc.\ have analogous meanings. We 
denote by $\overline{\hphantom{x}\vphantom{a}} : A \to A$ the involutive
automorphism such that $\overline{v^a}=v^{-a}$ for all $a\in\AC$. 

\section{Weight functions and cells} \label{sec1}

Let $p : W \to \AC$, $w \mapsto p_w$, be a {\it weight function} in the sense 
of Lusztig~\cite{Lusztig03}, that is, we have $p_s=p_t$ whenever $s,t\in S$ 
are conjugate in $W$, and $p_w=p_{s_1}+\cdots +p_{s_k}$ if $w=s_1\cdots s_k$
(with $s_i \in S$) is a reduced expression for $w\in W$.  The original setup
in \cite{KL} corresponds to the case where $\AC=\ZM$ and $p_s=1$ for all 
$s \in S$; this will be called the ``equal parameter case''.  We shall 
assume throughout that $p_s>0$ for all $s\in S$. (There are standard 
techniques for reducing the general case to this case~\cite[\S 2]{bo3}.)

Let $\bH=\bH_A(W,S,p)$ be the corresponding generic Iwahori--Hecke algebra.
This algebra is free over 
$A$ with basis $(T_w)_{w \in W}$, and the multiplication is given by the rule
\[ T_sT_w=\left\{\begin{array}{cl} T_{sw} & \quad \mbox{if $sw>w$},\\
T_{sw}+(v^{p_s}-v^{-p_s})T_w & \quad \mbox{if $sw<w$},
\end{array}\right.\]
where $s\in S$ and $w\in W$; here, $\leq$ denotes the Bruhat--Chevalley
order on $W$. 

Let $(C_w')_{w\in W}$ be the ``new'' basis of $\bH$ introduced in 
\cite[(1.1.c)]{KL}, \cite[\S 2]{Lusztig83}. (These basis elements are
denoted $c_w$ in \cite{Lusztig03}.) For any $x,y\in W$, we write 
\[ C_x'\,C_y'=\sum_{z\in W} h_{x,y,z} \, C_z' \qquad \mbox{where
$h_{x,y,z} \in A$ for all $x,y,z\in W$}.\]
We have the following more explicit formula for $s\in S$, $y\in W$
(see \cite[\S 6]{Lusztig83}, \cite[Chap.~6]{Lusztig03}):
\[C_s'\,C_y' = \left\{\begin{array}{ll} \displaystyle{(v^{p_s}+
v^{-p_s})\,C_y'} &\quad \mbox{if $sy<y$},\\ 
\displaystyle{C_{sy}'+ \sum_{z\in W\,:\,sz<z<y} 
M_{z,y}^s C_z'} &\quad \mbox{if $sy>y$},
\end{array}\right.\] 
where $C_s'=T_s+v^{-p_s}T_1$ and $M_{z,y}^s=\overline{M}_{z,y}^s \in A$ 
is determined as in \cite[\S 3]{Lusztig83}. 

As in \cite[\S 8]{Lusztig03}, we write $x \leftarrow_{L} y$ if there
exists some $s\in S$ such that $h_{s,y,x}\neq 0$, that is, $C_x'$ occurs 
with a non-zero coefficient in the expression of $C_s'\, C_y'$ in the
$C'$-basis. The Kazhdan--Lusztig left pre-order $\leq_{L}$ is the 
transitive closure of $\leftarrow_{L}$. The equivalence relation associated 
with $\leq_{L}$ will be denoted by $\sim_{L}$ and the corresponding 
equivalence classes are called the {\em left cells} of~$W$. Note that
$\bH C_w\subseteq \sum_{y} AC_y$ where the sum runs over all
$y\in W$ with $y\leq_{L} w$.

Similarly, we can define a pre-order $\leq_{R}$ by considering
multiplication by $C_s'$ on the right in the defining relation. The
equivalence relation associated with $\leq_{R}$ will be denoted by
$\sim_{R}$ and the corresponding equivalence classes are called the
{\em right cells} of $W$.  We have
\equat\label{eq:left-right}
x \leq_{R} y \quad \Longleftrightarrow \quad x^{-1} \leq_{L} y^{-1};
\endequat
see \cite[5.6, 8.1]{Lusztig03}. Finally, we define a pre-order $\leq_{LR}$ 
by the condition that $x\leq_{LR} y$ if there exists a sequence $x=x_0,x_1,
\ldots, x_k=y$ such that, for each $i \in \{1,\ldots,k\}$, we have 
$x_{i-1} \leq_{L} x_i$ or $x_{i-1}\leq_{R} x_i$. The equivalence
relation associated with $\leq_{LR}$ will be denoted by $\sim_{LR}$ 
and the corresponding equivalence classes are called the {\em two-sided 
cells} of $W$. 

\begin{defn} \label{defclosed} A (non-empty) subset $\Gamma$ of $W$ is 
called {\bfit left-closed} if, for any $x,y\in\Gamma$,
we have $\{z\in W\mid x\leq_L z\leq_L y\}\subseteq \Gamma$. 
\end{defn}

Note that any such subset is a union of left cells. A left cell itself 
is clearly left-closed with respect to $\leq_L$. It immediately follows
from these definitions that, given any left-closed subset $\Gamma
\subseteq W$, the $A$-submodules 
\eqna
{\cI}_{\Gamma} &=&\langle C_w'\mid w\leq_{L} 
z\mbox{ for some $z \in\Gamma$}\rangle_A,\\
\hat{\cI}_{\Gamma} &=&\langle C_w'\mid w \not\in \Gamma, w\leq_{L} z
\mbox{ for some $z \in\Gamma$}\rangle_A.
\endeqna
are left ideals in $\bH$. Hence we obtain an $\bH$-module $[\Gamma]_A:=
{\cI}_{\Gamma}/ \hat{\cI}_{\Gamma}$, which is free over $A$ with basis
given by $(e_x)_{x\in \Gamma}$, where $e_x$ denotes the residue class 
of $C_x'$ in $[\Gamma]_A$. The action of $C_w'$ ($w \in W$) is
given by the formula
\[ C_w'\cdot e_x=\sum_{y \in \Gamma} h_{w,x,y}\, e_y.\]

\section{Cells and parabolic subgroups}\label{sec:parabolic}

A key tool in this work will be the process of {\it induction of cells}.  
Let $I\subseteq S$ and consider the parabolic subgroup $W_I\subseteq W$ 
generated by $I$. Then
\[ X_I:=\{w\in W\mid ws>w \mbox{ for all $s\in I$}\}\]
is the set of distinguished left coset representatives of $W_I$ in $W$.
The map $X_I \times W_I \rightarrow W$, $(x,u) \mapsto xu$, is a bijection
and we have $\ell(xu)=\ell(x)+\ell(u)$ for all $x\in X_I$ and $u\in W_I$; see
\cite[\S 2.1]{gepf}. Thus, given $w\in W$, we can write uniquely $w=xu$
where $x\in X_I$ and $u\in W_I$. In this case, we denote $\prI_I(w):=u$.
Let $\leq_{L,I}$ and $\sim_{L,I}$ be respectively the pre-order and 
equivalence relations for $W_I$ defined similarly as $\leq_L$ and $\sim_L$ 
are defined in $W$.  

\begin{thm} \label{cellind} Let $I\subseteq S$. If $x$, $y \in W$ are such 
that $x \leq_L y$ (resp.\ $x \sim_L y$), then $\prI_I(x) \leq_{L,I} 
\prI_I(y)$ (resp.\ $\prI_I(x) \sim_{L,I} \prI_I(y)$). In particular, 
if $\Gamma$ is a left cell of $W_I$, then
$X_I\Gamma$ is a union of left cells of $W$.
\end{thm}

This was first proved by Barbasch--Vogan \cite[Cor.~3.7]{BV2} for finite 
Weyl groups in the equal parameter case (using the theory of primitive 
ideals); see \cite{myind} for the general case.

\begin{exmp} \label{cellind1} Let $\Gamma$ be a left-closed subset of $W_I$. 
Then the subset $X_I\Gamma$ of $W$ is left-closed (see 
Theorem~\ref{cellind}).  Let $\bH_I\subseteq \bH$ be the parabolic 
subalgebra spanned by all $T_w$ where
$w\in W_I$. Then we obtain the $\bH_I$-module $[\Gamma]_A$, with 
standard basis $(e_w)_{w\in\Gamma}$, and the $\bH$-module 
$[X_I\Gamma]_A$, with standard basis $(e_{xw})_{x\in X_I,w\in\Gamma}$.
By \cite[3.6]{myrel}, we have an isomorphism of $\bH$-modules
\[ [X_I\Gamma]_A\stackrel{\sim}{\rightarrow} \Ind_I^S([\Gamma]_A), 
\qquad e_{yv} \mapsto \sum_{x \in X_I,w\in \Gamma} p_{xu,yv}^*\,
\big(T_x \otimes e_u\big),\]
where $p_{xu,yv}^*\in A$ are the {\em relative} Kazhdan--Lusztig
polynomials of \cite[Prop.~3.3]{myind} and, for any $\bH_I$-module $V$, we 
denote by $\Ind_I^S(V):=\bH \otimes_{\bH_I} V$ the {\em induced module}, 
with basis $(T_x\otimes e_w)_{x\in X_I,w\in \Gamma}$. (In 
\cite[\S 3]{myrel}, it is not stated explicitly that $\Gamma=X_I\Gamma$ 
is left-closed, but this condition is used implicitly in the discussion 
there.)
\end{exmp}

A first invariant of left cells is given as follows. For any $w\in W$, we 
denote by $\RC(w):=\{s\in S\mid ws<w\}$ the {\em right descent set} of~$w$ 
(or {\it $\tau$-invariant} of $w$ in the language of primitive ideals; see
\cite{BV2}). The next result has been proved in~\cite[2.4]{KL} (for the 
equal parameter case) and in~\cite[8.6]{Lusztig03} (for the unequal parameter 
case).

\begin{prop}[Kazhdan-Lusztig, Lusztig]\label{klright} Let $x,y\in W$.
\begin{itemize} 
\itemth{a} If $x \leq_L y$ then $\RC(y) \subseteq \RC(x)$. 

\itemth{b} If $x\sim_{L} y$, then $\RC(x)=\RC(y)$. 

\itemth{c} For any $I\subseteq S$, the set $\{w\in W\mid \RC(w)=I\}$ is
a union of left cells of~$W$.
\end{itemize}
\end{prop}

We show how this can be deduced from Theorem~\ref{cellind}. First, note that 
(b) and (c) easily follow from (a), so we only need to prove (a). Let
$x,y\in W$ be such that $x \leq_L y$. Let $s\in \RC(y)$ and set $I=\{s\}$. 
Then $\prI_I(y)=s$ and so $\prI_I(x)\leq_{L,I} \prI_I(y)=s \in W_I=\{1,s\}$. 
Since $p_s>0$, the definitions immediately show that $s\leq_{I,L} 1$ but
$s\not\sim_{L,I} 1$. Hence, we must have $\prI_I(x) =s$ and so $s\in 
\RC(x)$. Thus, we have $\RC(y) \subseteq \RC(x)$. 

\section{Left cellular maps} \label{sec3}

\begin{defn} \label{leftc} A map $\delta :W\to W$ is
called {\bfit left cellular} if the following conditions are satisfied for 
every left cell $\Gamma \subseteq W$ (with respect to the given weights 
$\{p_s\mid s\in S\}$):
\begin{itemize}\itemindent0.5cm
\itemth{A1} $\delta(\Gamma)$ also is a left cell.
\itemth{A2} The map $\delta$ induces an $\bH$-module isomorphism 
$[\Gamma]_A\cong [\delta(\Gamma)]_A$.
\end{itemize}
\end{defn}

A prototype of such a map is given by the Kazhdan--Lusztig $*$-operations. 
We briefly recall how this works. For any $s,t\in S$ such that $st\neq ts$, 
we set 
\[ \DC_R(s,t):=\{w\in W \,|\, \RC(w)\cap \{s,t\} \text{ has exactly one
element}\}\]
and, for any $w\in \DC_R(s,t)$, we set $\cT_{s,t}(w):=\{ws,wt\}\cap 
\DC_R(s,t)$. (See \cite[\S 4]{KL}, \cite[\S 3]{voga}.) Note that
$\cT_{s,t}(w)$ consists of one or two elements; in order to have a uniform 
notation, we consider $\cT_{s,t}(w)$ as a multiset 
with two identical elements if $\{ws,wt\}\cap \DC_R(s,t)$ consists of only 
one element. 

If $st$ has order $3$, then the intersection $\{ws,wt\}\cap 
\DC_R(s,t)$ consists of only one element which will be denoted by $w^*$. 
Thus, we have $\cT_{s,t}(w)=\{w^*,w^*\}$ in this case.  With this notation,
we can now state:

\begin{prop}[Kazhdan--Lusztig $*$-operations \protect{\cite[\S 4]{KL}}] 
\label{starop} Assume that we are in the equal parameter case and that 
$st$ has order $3$. Then we obtain a left cellular map $\delta\colon W
\rightarrow W$ by setting 
\[\delta(w)=\left\{\begin{array}{cl} w^* & \quad \mbox{if $w\in 
\DC_R(s,t)$}, \\ w & \quad \mbox{otherwise}.\end{array}\right.\]
In particular, if $\Gamma \subseteq \DC_R(s,t)$ is a left cell, then 
so is $\Gamma^*:= \{w^* \mid w \in \Gamma\}$.
\end{prop}

(In Corollary~\ref{star1} below, we extend this to the unequal parameter
case.)

If $st$ has order $\geq 4$, then the set $\cT_{s,t}(w)$ may contain two 
distinct elements. In order to obtain a single-valued operator, Vogan 
\cite[\S 4]{voga1} (for the case $m=4$; see also McGovern \cite[\S 4]{mcg}) 
and Lusztig \cite[\S 10]{Lu1} (for any $m\geq 4$) propose an alternative
construction, as follows.

\begin{rem} \label{remstrings} Let $s,t\in S$ be such that $st$ has any
finite order $m\geq 3$. Let $W_{s,t}=\langle s,t\rangle$, a dihedral group 
of order $2m$. For any $w\in W$, the coset $wW_{s,t}$ can be partitioned 
into four subsets: one consists of the unique element $x$ of minimal length, 
one consists of the unique element of maximal length, one consists of the 
$(m-1)$ elements $xs,xst,xsts, \ldots$ and one consists of the $(m-1)$ 
elements $xt,xts, xtst,\ldots$. Following Lusztig \cite[10.2]{Lu1}, the 
last two subsets (ordered as above) are called {\em strings}. (Note that 
Lusztig considers the coset $W_{s,t}w$ but, by taking inverses, the two 
versions are clearly equivalent.) Thus, if $w\in \DC_R(s,t)$, then $w$ 
belongs to a unique string which we denote by $\lambda_w$; we certainly 
have $\cT_{s,t}(w)\subseteq\lambda_w \subseteq \DC_R(s,t)$ for all 
$w\in \DC_R(s,t)$. 

We define an involution $\DC_R(s,t)\rightarrow \DC_R(s,t)$, $w\mapsto 
\tilde{w}$, as follows. Let $w\in\DC_R(s,t)$ and $i\in\{1,\ldots, m-1\}$ 
be the index such that $w$ is the $i$th element of the string $\lambda_w$. 
Then $\tilde{w}$ is defined to be the $(m-i)$th element of $\lambda_{w}$. 
Note that, if $m=3$, then $\tilde{w}=w^*$, with $w^*$ as in 
Proposition~\ref{starop}. 
\end{rem}

Let us write $T_xT_y=\sum_{z\in W} f_{x,y,z}T_z$ where $f_{x,y,z}\in A$
for all $x,y,z\in W$. Following \cite[13.2]{Lusztig03}, we say that 
$\bH$ is {\bfit bounded} if there exists some positive $N\in \AC$ such that
$v^{-N}f_{x,y,z}\in A_{\leq 0}$ for all $x,y,z\in W$. We can now state:

\begin{prop}[Lusztig \protect{\cite[10.7]{Lu1}}] \label{lu107}
Assume that we are in the equal parameter case and that $\bH$ is 
bounded. If $\Gamma \subseteq \DC_R(s,t)$ is a left cell, then so is 
$\tilde{\Gamma}:=\{\tilde{w} \mid w \in \Gamma\}$.
\end{prop}

(It is assumed in [{\em loc.\ cit.}] that $W$ is crystallographic, but 
this assumption is now superfluous thanks to Elias--Williamson \cite{EW}. 
The boundedness assumption is obviously satisfied for all finite Coxeter 
groups. It also holds, for example, for affine Weyl groups; see the
remarks following \cite[13.4]{Lusztig03}.) 

In Corollary~\ref{star1} below, we shall show that $w\mapsto \tilde{w}$ 
also gives rise to a left cellular map and that this works without any 
assumption, as long as $p_s=p_t$. 

\section{Vogan's invariants} \label{seceq}

\begin{svgraybox}{{\bf Hypothesis.} {\it 
Throughout this section, and only in this section, 
we assume that we are in the equal parameter case.}}
\end{svgraybox}

We recall the following definition.
\begin{defn}[Vogan \protect{\cite[3.10, 3.12]{voga}}] \label{defvog} 
For $n \ge 0$, we define an equivalence relation $\approx_n$ on $W$ 
inductively as follows. Let $x,y \in W$.
\begin{itemize}
\item[$\bullet$] For $n=0$, we write $x \approx_0 y$ if $\RC(x)=\RC(y)$. 
\item[$\bullet$] For $n\ge 1$, we write $x \approx_n y$ if $x\approx_{n-1} y$ 
and if, for any $s,t \in S$ such that $x$, $y \in \DC_R(s,t)$ (where $st$ 
has order $3$ or $4$), the following holds: if $\cT_{s,t}(x)=\{x_1,x_2\}$ 
and $\cT_{s,t}(y) =\{y_1,y_2\}$, then either $x_1\approx_{n-1} y_1$, 
$x_2\approx_{n-1}y_2$ or $x_1\approx_{n-1} y_2$, $x_2\approx_{n-1} y_1$.
\end{itemize}
If $x \approx_n y$ for all $n\geq 0$, then $x$, $y$ are said to 
have the same {\bfit generalised $\tau$-invariant}. 
\end{defn}

The following result was originally formulated and proved for finite 
Weyl groups by Vogan \cite[\S 3]{voga}, in the language of primitive 
ideals in enveloping algebras. It then follows for cells as defined in 
Section~\ref{sec1} using a known dictionary (see, e.g., Barbasch--Vogan 
\cite[\S 2]{BV2}). The proof in general relies on Proposition~\ref{starop} 
and results on {\it strings} as defined in Remark~\ref{remstrings}.

\begin{thm}[Kazhdan--Lusztig \protect{\cite[\S 4]{KL}}, Lusztig 
\protect{\cite[\S 10]{Lu1}}, Vogan \protect{\cite[\S 3]{voga}}] 
\label{vogan1} Assume that $\bH$ is bounded and recall that we are in the 
equal parameter case. Let $\Gamma$ be a left cell of $W$. Then all elements 
in $\Gamma$ have the same generalised $\tau$-invariant.
\end{thm}

\begin{proof}\smartqed We prove by induction on $n$ that, if $y,w\in W$
are such that $y\sim_L w$, then $y\approx_n w$. For $n=0$, this holds by 
Propositon~\ref{klright}. Now let $n>0$. By induction, we already know 
that $y\approx_{n-1} w$. Then it remains to consider $s,t \in S$ such 
that $st\neq ts$ and $y,w\in\DC_R(s,t)$.  If $st$ has order $3$, then 
$\cT_{s,t}(y)=\{y^*,y^*\}$ and $\cT_{s,t}(w)=\{w^*,w^*\}$; furthermore, 
by Proposition~\ref{starop}, we have $y^*\sim_L w^*$ and so 
$y^*\approx_{n-1} w^*$, by induction. Now assume that $st$ has order~$4$. 
In this case, the argument is more complicated (as it is also in the 
setting of \cite[\S 3]{voga}.) Let $I=\{s,t\}$ and $\Gamma$ be 
the left cell containing $y,w$. Since all elements in $\Gamma$ have the
same right descent set (by Proposition~\ref{klright}), we can choose the
notation such that $xs<x$ and $xt>x$ for all $x\in \Gamma$. Then, for 
$x\in \Gamma$, we have $x=x's$, $x=x'ts$ or $x=x'sts$ where $x'\in X_I$. 
This yields that 
$$
\cT_{s,t}(x)=\left\{\begin{array}{cl} 
\{x'st,x'st\}  & \qquad \mbox{if $x=x's$},\\
\{x't,x'tst\}  & \qquad \mbox{if $x=x'ts$},\\
\{x'st,x'st\}  & \qquad \mbox{if $x=x'sts$}.\end{array}\right.
\leqno{(\dagger)}
$$
We now consider the string $\lambda_x$ and distinguish two cases.

\medskip
\noindent{\bfit Case 1}. Assume that there exists some $x\in \Gamma$ such 
that $x=x's$ or $x=x'sts$. Then $\lambda_x=(x's,x'st,x'sts)$ and so the 
set $\Gamma^*:=\bigl(\bigcup_{w\in \Gamma} \lambda_w \bigr)\setminus 
\Gamma$ contains elements with different right descent sets. On the other 
hand, by \cite[Prop.~10.7]{Lu1}, $\Gamma^*$ is the union of at most two 
left cells. (Again, the assumption in [{\em loc.\ cit.}] that $W$ is 
crystallographic is now superfluous thanks to \cite{EW}.) We conclude 
that $\Gamma^*=\Gamma_1\cup\Gamma_2$ where $\Gamma_1$, $\Gamma_2$ are left 
cells such that:  
\begin{itemize}
\item all elements in $\Gamma_1$ have $s$ in their right descent set, but
not $t$;
\item all elements in $\Gamma_2$ have $t$ in their right descent set, but
not $s$.
\end{itemize}
Now consider $y,w \in\Gamma$; we write $\cT_{s,t}(y)=\{y_1, y_2\}\subseteq 
\Gamma^*$ and $\cT_{s,t}(w)=\{w_1,w_2\}\subseteq \Gamma^*$. By 
($\dagger$), all the elements $y_1,y_1,w_1,w_2$ belong to $\Gamma_2$. In 
particular, $y_1\sim_L w_1$, $y_2\sim_L w_2$ and so, by induction, 
$y_1\approx_{n-1} w_1$, $y_2\approx_{n-1} w_2$. 

\medskip
\noindent{\bfit Case 2}. We are not in Case~1, that is, all elements 
$x\in\Gamma$ have the form $x=x'ts$ where $x'\in X_I$. Then $\lambda_x=
(x't,x'ts, x'tst)$ for each $x\in\Gamma$. Let us label the elements in such
a string as $x_1,x_2,x_3$. Then $x=x_2$ and $\cT_{s,t}(x)=\{x't,x'tst\}=
\{x_1,x_3\}$. 

Now consider $y,w\in\Gamma$. There is a chain of elements
which connect $y$ to $w$ via the elementary relations $\leftarrow_L$, and 
vice versa. Assume first that $y,w$ are directly connected as 
$y\leftarrow_L w$. Using the labelling $y=y_2$, $w=w_2$ and the notation 
of \cite[10.4]{Lu1}, this means that $a_{22}\neq 0$. Hence, the identities 
``$a_{11}=a_{33}$'', ``$a_{13}=a_{31}$'', ``$a_{22}=a_{11}+a_{13}$'' in
\cite[10.4.2]{Lu1} imply that 
\[(y_1\leftarrow_L w_1 \mbox{ and } y_3 \leftarrow_L w_3) \quad 
\mbox{or} \quad (y_1\leftarrow_L w_3\mbox{ and } y_3\leftarrow_L w_1).\]
Now, in general, there is a sequence of elements $y=y^{(0)},y^{(1)},\ldots,
y^{(k)}=w$ in $\Gamma$ such that $y^{(i-1)} \leftarrow_L y^{(i)}$ for 
$1\leq i\leq k$. At each step, the elements in the strings corresponding 
to these elements are related as above. Combining these steps, one easily 
sees that 
\[ (y_1\leq_L w_1 \mbox{ and } y_3\leq_L w_3) \quad \mbox{or} \quad
(y_1\leq_L w_3 \mbox{ and } y_3\leq_L w_1).\]
(See also \cite[Prop.~4.6]{shi}.) Now, all elements in a string belong to 
the same right cell (see \cite[10.5]{Lu1}); in particular, all the elements 
$y_i,w_j$ belong to the same two-sided cell. Hence, \cite[Cor.~6.3]{Lu1} 
implies that either $y_1 \sim_L w_1$, $y_3\sim_L w_3$ or $y_1\sim_L w_3$, 
$y_3\sim_L w_1$. (Again, the assumption in [{\em loc.\ cit.}] that $W$ is 
crystallographic is now superfluous thanks to \cite{EW}.) Consequently, 
by induction, we have either $y_1\approx_{n-1} w_1$, $y_3\approx_{n-1}w_3$ 
or $y_1\approx_{n-1} w_3$, $y_3\approx_{n-1}w_1$.
\qed\end{proof}

One of the most striking results about this invariant has been obtained
by Garfinkle \cite[Theorem~3.5.9]{gar3}: two elements of a Weyl group of 
type $B_n$ belong to the same left cell (equal parameter case) if and 
only if the elements have the same generalised $\tau$-invariant. This
fails in general; a counter-example is given by $W$ of type $D_n$ for
$n\geq 6$ (as mentioned in the introduction of \cite{gar1}).

\begin{rem} \label{rem11} Vogan \cite[\S 4]{voga1} also proposed
the following modification of the above invariant. Let $s,t\in S$ be such 
that $st$ has finite order $m\geq 3$. Let us set $\tilde{\cT}_{s,t}(w)
:=\{\tilde{w}\}$ for any $w\in \DC_R(s,t)$, with $\tilde{w}$ as in 
Remark~\ref{remstrings}. Then we obtain a new invariant by exactly 
the same procedure as in Definition~\ref{defvog}, but using 
$\tilde{\cT}_{s,t}$ instead of $\cT_{s,t}$ and allowing any $s,t\in S$ 
such that $st$ has finite order~$\geq 3$. (Note that Vogan only
considered the case where $m=4$, but then Lusztig's method of strings
shows how to deal with the general case.) In any case, this is the
model for our more general construction of invariants below.
\end{rem}

\section{Induction of left cellular maps} \label{sec2}

We return to the general setting of Section~\ref{sec1}, where 
$\{p_s\mid s\in S\}$ are any positive weights for $W$. 

\begin{defn} \label{mydef1} A pair $(I,\delta)$ consisting of a subset
$I\subseteq S$ and a left cellular map $\delta \colon W_I\rightarrow W_I$ 
is called {\bfit KL-admissible}. We recall that this means that the 
following conditions are satisfied for every left cell $\Gamma \subseteq 
W_I$ (with respect to the weights $\{p_s\mid s\in I\}$):
\begin{itemize}\itemindent0.5cm
\itemth{A1} $\delta(\Gamma)$ also is a left cell.
\itemth{A2} The map $\delta$ induces an $\bH_I$-module isomorphism 
$[\Gamma]_A\cong [\delta(\Gamma)]_A$.
\end{itemize}
We say that $(I,\delta)$ is {\bfit strongly KL-admissible} if, in addition 
to $(\Arm 1)$ and $(\Arm 2)$, the following condition is satisfied:
\begin{itemize}\itemindent0.5cm
\itemth{A3} We have $u\sim_{R,I}\delta(u)$ for all $u\in W_I$. 
\end{itemize}
If $I \subseteq S$ and if $\d : W_I \to W_I$ is a map, we obtain a map 
$\d^L : W \to W$ by
$$\d^L(xw)=x\d(w)\qquad \mbox{for all $x \in X_I$ and $w \in W_I$}.$$
The map $\d^L$ will be called the {\it left extension} of $\d$ to $W$.
\end{defn}

\begin{thm} \label{myprop} Let $(I,\delta)$ be a KL-admissible pair. Then 
the following hold.
\begin{itemize}
\itemth{a} The left extension $\d^L:W\to W$ is a left cellular map for $W$.  
\itemth{b} If $(I,\delta)$ is strongly admissible, then we 
have $w\sim_R \d^L(w)$ for all $w\in W$.
\end{itemize}
\end{thm}

\begin{proof}\smartqed (a) By Theorem~\ref{cellind}, there is a left cell 
$\Gamma'$ of $W_I$ such that $\Gamma\subseteq X_I\Gamma'$. By condition (A1) 
in Definition~\ref{mydef1}, the set $\Gamma_1':=\delta(\Gamma')$ also is a 
left cell of $W_I$ and, by condition (A2), the map $\delta$ induces an
$\bH_I$-module isomorphism $[\Gamma']_A\cong [\Gamma_1']_A$. By 
Example~\ref{cellind1}, the subsets $X_I\Gamma'$ and $X_I\Gamma_1'$
of $W$ are left-closed and, hence, we have corresponding $\bH$-modules 
$[X_I\Gamma']_A$ and $[X_I\Gamma_1']_A$. These two $\bH$-modules are 
isomorphic to the induced modules $\Ind_I^S([\Gamma'])$ and 
$\Ind_I^S([\Gamma_1'])$, respectively, where explicit isomorphisms are 
given by the formula in Example~\ref{cellind1}. Now, by
\cite[Lemma~3.8]{myrel}, we have
\[ p_{xu,yv}^*=p_{xu_1,yv_1}^* \qquad \mbox{for all $x,y\in X_I$ and
$u,v\in \Gamma'$},\]
where we set $u_1=\delta(u)$ and $v_1=\delta(v)$ for $u,v\in\Gamma'$. 
By \cite[Prop.~3.9]{myrel}, this implies that $\d^L$ maps the 
partition of $X_I\Gamma'$ into left cells of $W$ onto the analogous 
partition of $X_I\Gamma_1'$. In particular, since $\Gamma\subseteq 
X_I\Gamma'$, the set $\d^L(\Gamma) \subseteq X_I\Gamma_1'$ 
is a left cell of $W$; furthermore, \cite[Prop.~3.9]{myrel} also shows
that $\d^L$ induces an $\bH$-module isomorphism
$[\Gamma]_A\cong [\d^L(\Gamma)]_A$. 

(b) Since condition (A3) in Definition~\ref{mydef1} is assumed to hold, this
is just a restatement of \cite[Prop.~9.11(b)]{Lusztig03}.
\qed\end{proof}

We will now give examples in which $|I|=2$. Let us first fix some notation. 
If $s$, $t \in S$ are such that $s\neq t$ and $st$ has finite order, let 
$w_{s,t}$ denote the longest element of $W_{s,t}=\langle s,t \rangle$ and let
\begin{align*}
\G_s^{s,t}&=\{w \in W_{s,t}~|~\ell(ws) < \ell(w)~\text{and}~\ell(wt) > 
\ell(w)\}, \\
\G_t^{s,t}&=\{w \in W_{s,t}~|~\ell(ws) > \ell(w)~\text{and}~\ell(wt) < 
\ell(w)\}.
\end{align*}

\begin{exmp}[Dihedral groups with equal parameters]
\label{ex:eq}
Let $s$, $t \in S$ be such that $p_s=p_t$ and $s \neq t$. It follows 
from~\cite[\S{8.7}]{Lusztig03} that $\{1\}$, $\{w_{s,t}\}$, $\G_s^{s,t}$ 
and $\G_t^{s,t}$ are the left cells of $W_{s,t}$. Let $\s_{s,t}$ be the 
unique group automorphism of $W_{s,t}$ which exchanges $s$ and $t$. 
Now, let $\d_{s,t}$ denote the map $W_{s,t} \mapsto W_{s,t}$ 
defined by 
$$\d_{s,t}(w)=
\left\{\begin{array}{lcl}
w && \text{if $w \in \{1,w_{s,t}\}$,}\\
\s_{s,t}(w) w_{s,t} &~& \text{otherwise.} \\
\end{array}\right.$$
Then, by~\cite[Lemma~7.2~and~Prop.~7.3]{Lusztig03}, the pair $(\{s,t\},
\d_{s,t})$ is strongly KL-admissible. Therefore, by Theorem~\ref{myprop}, 
\begin{center}
$\d_{s,t}^L\colon W\rightarrow W$ is a left cellular map.
\end{center} 
In particular, this means:
\equat\label{eq:star}
x \sim_L y ~\text{\it if and only if}~\d_{s,t}^L(x) \sim_L \d_{s,t}^L(y)
\endequat
for all $x$, $y \in W$. Note also the following facts:
\begin{itemize}
\item[$\bullet$] If $st$ has odd order, then $\d_{s,t}$ exchanges the left 
cells $\G_s^{s,t}$ and $\G_t^{s,t}$.

\item[$\bullet$] If $st$ has even order, then $\d_{s,t}$ stabilizes the 
left cells $\G_s^{s,t}$ and $\G_t^{s,t}$. 
\end{itemize}
For example, if $st$ has order $3$, then $\G_s^{s,t}=\{s,ts\}$ 
and $\G_t^{s,t}=\{t,st\}$; furthermore, $\d_{s,t}(s) =st$ and $\d_{s,t}
(ts)=t$. The matrix representation afforded by $[\G_s^{s,t}]_A$ with 
respect to the basis $(e_{s}, e_{ts})$ is given by:
\[ C_{s}' \mapsto \left[\begin{array}{cc} v^{p_s}+v^{-p_s} & 1 \\ 0 & 0
\end{array}\right],\qquad C_{t}' \mapsto \left[\begin{array}{cc} 0 & 0
\\ 1 & v^{p_t}+v^{-p_t} \end{array}\right] \qquad (p_s=p_t). \]
The fact that $\d_{s,t}$ is left cellular just means that we obtain exactly 
the same matrices when we consider the matrix representation afforded 
by $[\G_{s,t}^t]_A$ with respect to the basis $(e_{st},e_{t})$.
\end{exmp}

Let us explicitly relate the above discussion to the $*$-operations in 
Proposition~\ref{starop} and the extension in Proposition~\ref{lu107}. 
In particular, this yields new proofs of these two propositions and 
shows that they also hold in the unequal parameter case, without any 
further assumptions, as long as $p_s=p_t$. (Partial results in this 
direction are obtained in \cite[Cor.~3.5(4)]{shi1}.) 

\begin{cor} \label{star1} Let $s,t\in S$ be such that $st$ has finite 
order $\geq 3$ and assume that $p_s=p_t$. Then, with the notation in 
Remark~\ref{remstrings}, we obtain a left cellular map $\delta\colon W
\rightarrow W$ by setting
\[\delta(w)=\left\{\begin{array}{cl} \tilde{w} & \quad \mbox{if $w\in 
\DC_R(s,t)$}, \\ w & \quad \mbox{otherwise}.\end{array}\right.\]
(If $st$ has order $3$, then this coincides with the map defined in 
Proposition~\ref{starop}.)
\end{cor}

\begin{proof} \smartqed Just note that, if $w \in \DC_R(s,t)$, then 
$\delta_{s,t}^L(w)=\tilde{w}$. Thus, the assertion simply is a restatement 
of the results in Example~\ref{ex:eq}. Furthermore, if $st$ has order $3$, 
then $\tilde{w}=w^*$ for all $w \in \DC_R(s,t)$, as noted
in Remark~\ref{remstrings}.
\qed\end{proof}

\begin{exmp}[Dihedral groups with unequal parameters]
\label{ex:uneq} Let $s, t \in S$ be such that $st$ has \emph{even} order 
$\ge 4$ and that $p_s < p_t$. Then it follows from~\cite[\S{8.7}]{Lusztig03} 
that $\{1\}$, $\{w_{s,t}\}$, $\{s\}$, $\{w_{s,t}s\}$, $\G_s^{s,t}
\setminus\{s\}$ and $\G_t^{s,t}\setminus\{w_{s,t}s\}$ are the left cells 
of $W_{s,t}$. Now, let $\d_{s,t}$ denote the map $W_{s,t} \mapsto W_{s,t}$ 
defined by
$$\d_{s,t}^<(w)=
\left\{\begin{array}{lcl}
w && \text{if $w \in \{1,w_{s,t},s,w_{s,t}s\}$,}\\
ws &~& \text{otherwise.} \\
\end{array}\right.$$
Then, again by~\cite[Lemma~7.5~and~Prop.~7.6]{Lusztig03} (or 
\cite[Exc.~11.4]{gepf}), the pair $(\{s,t\},\d_{s,t}^<)$ is strongly 
KL-admissible. Therefore, again by Theorem~\ref{myprop}, $\d_{s,t}^{<,L}
\colon W\rightarrow W$ is a left cellular map. In particular, this means: 
\equat\label{eq:star-uneq}
x \sim_L y ~\text{\it if and only if}~\d_{s,t}^{<,L}(x) \sim_L 
\d_{s,t}^{<,L}(y)
\endequat
for all $x$, $y \in W$. Note also that $\d_{s,t}^<$ exchanges the left cells 
$\G_s^{s,t}\setminus\{s\}$ and $\G_t^{s,t}\setminus\{w_{s,t}s\}$ while it 
stabilizes all other left cells in $W_{s,t}$. 

For example, if $st$ has order $4$, then $\Gamma_1:=\G_s^{s,t}\setminus
\{s\}=\{ts,sts\}$ and $\Gamma_2:=\G_t^{s,t} \setminus\{w_{s,t}s\}=\{t,st\}$;
furthermore, $\d_{s,t}^{<}(ts)=t$ and $\d_{s,t}^{<}(sts)=st$. As before,
the fact that $\delta_{s,t}^{<}$ is left cellular just means that the matrix 
representation afforded by $[\Gamma_1]_A$ with respect to the basis 
$(e_{ts}, e_{sts})$ is exactly the same as the matrix representation 
afforded by $[\Gamma_2]_A$ with respect to the basis $(e_{t}, e_{st})$.
\end{exmp}

The next example shows that left extensions from dihedral subgroups are, in
general, not enough to describe all left cellular maps.


\begin{exmp}\label{ex:non-dihedral} Let $W$ be a Coxeter group of type 
$B_r$ ($r \geq 2$), with diagram and weight function as follows:
\begin{center}
\begin{picture}(250,18)
\put( 10,6){$B_r$}
\put( 50,6){\circle*{5}}
\put( 48,13){$b$}
\put( 50,6){\line(1,0){20}}
\put( 58,9){$\scriptstyle{4}$}
\put( 70,6){\circle*{5}}
\put( 68,13){$a$}
\put( 70,6){\line(1,0){30}}
\put( 90,6){\circle*{5}}
\put( 88,13){$a$}
\put(110,6){\circle*{1}}
\put(120,6){\circle*{1}}
\put(130,6){\circle*{1}}
\put(140,6){\line(1,0){10}}
\put(150,6){\circle*{5}}
\put(147,13){$a$}
\put(185,5){$b>(r-1)a>0$}
\end{picture}
\end{center}
This is the {\bfit asymptotic case} studied by Iancu and the first-named
author \cite{bo2}, \cite{BI}. In this case, the left, right and two-sided 
cells are described in terms of a Robinson--Schensted correspondence for 
bi-tableaux. Using results from [{\em loc.\ cit.}], it is shown in 
\cite[Theorem~6.3]{myrel} that the following hold:
\begin{itemize}
\item[(a)] If $\G_1$ and $\G_2$ are two left cells contained in the 
same two-sided cell, then there exists a bijection $\d : \G_1 \longiso \G_2$ 
which induces an isomorphism of $\bH$-modules $[\G_1]_A \longiso [\G_2]_A$. 
\item[(b)] The bijection $\d$ in (a) is uniquely determined by the 
condition that $w,\delta(w)$ lie in the same right cell.
\end{itemize}
However, one can check that, for $r \in \{3,4,5\}$, the map $\d$ does not
always arise from a left extension of a suitable left cellular map of a
dihedral subgroup of $W$. It is probable that this observation holds for 
any $r \ge 3$.
\end{exmp}

\begin{exmp} \label{affine} Let $W$ be an affine Weyl group and $W_0$
be the finite Weyl group associated with $W$. Then there is a well-defined
``lowest'' two-sided cell, which consists of precicely $|W_0|$ left
cells; see Guilhot \cite{guil} and the references there. It is likely
that these $|W_0|$ left cells are all related by suitable left cellular
maps.
\end{exmp}

\section{An extension of the generalised $\tau$-invariant} \label{sec6}


\begin{svgraybox}{{\bf Notation.} {\it We fix in this section a set $\D$ of 
KL-admissible pairs, as well as a surjective map $\r : W \to E$ (where 
$E$ is a fixed set) such that the fibers of $\r$ are unions of left cells. 
We then denote by $\VC_\D$ the group of permutations of $W$ 
generated by the family $(\d^L)_{(I,\d) \in \D}$.}}
\end{svgraybox}

Note that giving a surjective map $\r$ as above is equivalent to giving 
an equivalence relation on $W$ which is coarser than $\sim_L$. 

\medskip

Then, each $w \in W$ defines a map $\tau^{\D,\r}_w : \VC_\D \longto E$ 
as follows:
$$\tau_w^{\D,\r}(\s)=\r(\s(w))\qquad \mbox{for all $\s \in \VC_\D$}.$$

\begin{defn}\label{def:tau}
Let $x$, $y \in W$. We say that $x$ and $y$ have the same {\bfit  
$\tau^{\D,\r}$-invariant} if $\tau_x^{\D,\r}=\tau_y^{\D,\r}$ (as maps 
from $\VC_\D$ to $E$). The equivalence classes for this relation are 
called the {\bfit left Vogan ${\boldsymbol{(\D,\r)}}$-classes}.
\end{defn}

An immediate consequence of Theorem~\ref{myprop} is the following:

\begin{thm}\label{theo:vogan}
Let $x$, $y \in W$ be such that $x \sim_L y$. Then $x$ and $y$ have the same 
$\tau^{\D,\r}$-invariant.
\end{thm}

\begin{rem}\label{rem:tau} There is an equivalent formulation of 
Definition~\ref{def:tau} which is more in the spirit of Vogan's
Definition~\ref{defvog}. We define by induction on $n$ a family of 
equivalence relations $\approx_n^{\D,\r}$ on $W$ as follows. Let $x, y 
\in W$. 
\begin{itemize}
\item[$\bullet$] For $n=0$, we write $x \approx_0^{\D,\r} y$ if $\r(x)=\r(y)$.

\item[$\bullet$] For $n\geq 1$, we write $x \approx_n^{\D,\r} y$ if 
$x \approx_{n-1}^{\D,\r} y$ and $\d^L(x) \approx_{n-1}^{\D,\r} \d^L(y)$ for 
all $(I,\d) \in \D$.
\end{itemize}
Note that the relation $\approx_n^{\D,\r}$ is finer than 
$\approx_{n-1}^{\D,\r}$. It follows from the definition that $x,y$ 
have the same $\tau^{\D,\r}$-invariant if and only if $x \approx_n^{\D,\r} 
y$ for all $n \ge 0$.

This inductive definition is less easy to write than 
Definition~\ref{def:tau}, but it is more efficient for computational 
purpose. Indeed, if one finds an $n_0$ such that the relations 
$\approx_{n_0}^{\D,\r}$ and $\approx_{n_0+1}^{\D,\r}$ coincide, then 
$x$ and $y$ have the same $\tau^{\D,\r}$-invariant precisely when 
$x \approx_{n_0}^{\D,\r} y$. 
Note that such an $n_0$ always exists if $W$ is finite. Also, 
even in small Coxeter groups, the group $\VC_\D$ can become enormous 
(see Example~\ref{ex:h4} below) while $n_0$ is reasonably small and 
the relation $\approx_{n_0}^{\D,\r}$ can be computed quickly.
\end{rem}

\begin{exmp}[Enhanced right descent set]\label{ex:enhanced}
One could take for $\r$ the map $\RC : W \to \PC(S)$ (power set of $S$); 
see Proposition~\ref{klright}. Assuming that we are in the equal parameter
case, we then obtain exactly the invariant in Remark~\ref{rem11}. In the 
unequal parameter case, we can somewhat refine this, as follows. Let
$$S^p = S \cup \{sts~|~s,t \in S\text{ such that }p_s < p_t\}$$
and, for $w \in W$, let
$$\RC^p(w)=\{s \in S^p~|~\ell(ws) < \ell(w)\} \subseteq S^p.$$
Then it follows from the description of left cells of $W_{s,t}$ in 
Example~\ref{ex:uneq} and from Theorem~\ref{cellind} (by using the same 
argument as for the proof of Proposition~\ref{klright} given 
in~\S\ref{sec:parabolic}) that 
$$\text{\it if $x \leq_L y$, then $\RC^p(y) \subseteq \RC^p(x)$.}$$
In particular,
$$\text{\it if $x \sim_L y$, then $\RC^p(x) = \RC^p(y)$.}$$
So one could take for $\r$ the map $\RC^p : W \to \PC(S^p)$.
\end{exmp}

Let $\D_2$ be the set of all pairs $(I,\d)$ such that $I=\{s,t\}$
with $s\neq t$ and $p_s \le p_t$; furthermore, if $p_s=p_t$, then 
$\d = \d_{s,t}$ (as defined in Example~\ref{ex:eq}) while, if $p_s < p_t$, 
then $\d=\d_{s,t}^<$ (as defined in Example~\ref{ex:uneq}). Then the pairs 
in $\D_2$ are all strongly KL-admissible. With the notation in 
Example~\ref{ex:enhanced}, we propose the following conjecture:

\begin{conj}\label{myconj}
Let $x$, $y \in W$. Then $x \sim_L y$ if and only if $x \sim_{LR} y$ and 
$x,y$ have the same $\t^{\D_2,\RC^p}$-invariant.
\end{conj}

If $W$ is finite and we are in the equal parameter case, then 
Conjecture~\ref{myconj} is known to hold except possibly in type $B_n,D_n$;
see the remarks at the end of \cite[\S 6]{geha}. We have checked that
the conjecture also holds for $F_4$, $B_n$ ($n\leq 7$) and all 
possible weights, using {\sf PyCox} \cite{pycox}. 

By considering collections $\Delta$ with subsets $I\subseteq S$ of 
size bigger than~$2$, one can obtain further refinements of the above 
invariants. In particular, it is likely that the results 
of~\cite{bo2},~\cite{BI} can be interpreted in terms of generalised 
$\t^{\Delta,\rho}$-invariants for suitable $\Delta,\rho$. This will be
discussed elsewhere.

\begin{exmp}\label{ex:h4} Let $(W,S)$ be of type $H_4$. Then it can be 
checked by using computer computations in {\tt GAP} that 
$$|\VC_{\D_2}|=2^{40}\cdot 3^{20} \cdot 5^8 \cdot 7^4 \cdot 11^2.$$
On the other hand, the computation of left Vogan $(\D_2,\RC^p)$-classes using 
the alternative definition given in Remark~\ref{rem:tau} takes only a few 
minutes on a standard computer.
\end{exmp}

\begin{acknowledgement}
The first author is partly supported by the ANR (Project No.\
ANR-12-JS01-0003-01 ACORT). The second author is partly supported
by the DFG (Grant No.\ GE 764/2-1).
\end{acknowledgement}

\end{document}